\documentclass[a4paper,11pt]{amsart}

\usepackage{amsmath,amsfonts,amssymb,amscd,amsthm,amsbsy,upref}
\usepackage{verbatim}
\usepackage{amssymb,amsthm,amsmath,enumerate}
\usepackage{amsrefs}

\numberwithin{equation}{section}

\usepackage{enumerate}

\usepackage{amsrefs}
\usepackage{amsthm}
\usepackage{amsmath}		
\usepackage{amssymb}	
\usepackage{amsfonts}		

\usepackage{dsfont}

\usepackage{tikz}
\usepgflibrary{shadings}
\usetikzlibrary{arrows,shapes,positioning}

\usepackage[all]{xy}

\newcommand{\dk}{d_{\K}^k}

\newcommand{\K}{\mathbb{K}}
\newcommand{\N}{\mathbb{N}}

\newcommand{\R}{\mathbb{R}}

\newcommand{\M}{\mathbb{M}}

\newcommand{\J}{\mathcal{J}}
\newcommand{\B}{\mathcal{B}}
\newcommand{\T}{\mathcal{T}}

\newcommand{\ep}{\varepsilon}

\newcommand{\norm}[1]{\|#1\|}

\newcommand{\set}[1]{\left\{#1\right\}}

\newcommand{\abs}[1]{\lvert#1\rvert}

\newcommand{\indicator}[1]{{\mathbf 1}_{{#1}}}

\newtheorem{theorem}{Theorem}[section]

\theoremstyle{plain}

\newtheorem{lemma}[theorem]{Lemma}
\newtheorem{lem}[theorem]{Lemma}
\newtheorem{thm}[theorem]{Theorem}

\newtheorem{prop}[theorem]{Proposition}

\newtheorem{cor}[theorem]{Corollary}

\theoremstyle{definition}

\newtheorem{remark}[theorem]{Remark}

\newtheorem{defi}[theorem]{Definition}

\newcommand{\eps}{\varepsilon}

\newcommand{\car}{\mathds{1}}

\newcommand{\supp}{\text{supp}\,}
\newcommand{\diam}{\text{diam}\,}

\newcommand{\n}{\overline n}
\newcommand{\m}{\overline m}
\newcommand{\lbar}{\overline \ell}

\begin{document}

\title[On the coarse geometry of James spaces]{On the coarse geometry of James spaces}

\author{G.~Lancien}
\address{Gilles Lancien, Laboratoire de Math\'ematiques de Besan\c con, Universit\'e Bourgogne Franche-Comt\'e, CNRS UMR-6623, 16 route de Gray, 25030 Besan\c con C\'edex, Besan\c con, France}
\email{gilles.lancien@univ-fcomte.fr}

\author{C.~Petitjean}
\address{Colin Petitjean, Laboratoire de Math\'ematiques de Besan\c con, Universit\'e Bourgogne Franche-Comt\'e, CNRS UMR-6623, 16 route de Gray, 25030 Besan\c con C\'edex, Besan\c con, France}
\email{colin.petitjean@univ-fcomte.fr}

\author{A.~Proch\'azka}
\address{Antonin Proch\'azka, Laboratoire de Math\'ematiques de Besan\c con, Universit\'e Bourgogne Franche-Comt\'e, CNRS UMR-6623, 16 route de Gray, 25030 Besan\c con C\'edex, Besan\c con, France}
\email{antonin.prochazka@univ-fcomte.fr}

\thanks{The three authors are supported by the French ``Investissements d'Avenir'' program, project ISITE-BFC (contract
 ANR-15-IDEX-03).}

\begin{abstract}
In this note we prove that the Kalton interlaced graphs
do not equi-coarsely embed into the James space $\J$ nor into its dual $\J^*$. It is a particular case of a more general result on the non equi-coarse embeddability of the Kalton graphs into quasi-reflexive spaces with a special asymptotic stucture. This allows us to exhibit a coarse invariant for Banach spaces, namely the non equi-coarse embeddability of this family of graphs, which is very close to but different from the celebrated property $\mathcal Q$ of Kalton. We conclude with a remark on the coarse geometry of the James tree space $\J\T$ and of its predual.
\end{abstract}

\maketitle

\section{Introduction}

In a fundamental paper on the coarse geometry of Banach spaces (\cite{kaltonpropertyq}), N. Kalton introduced a property of metric spaces that he named property $\mathcal Q$. In particular,
its absence served as an obstruction to coarse embeddability into reflexive Banach spaces. This property is related to the behavior of Lipschitz maps defined on a particular family of metric graphs that we shall denote $([\N]^k,d_{\K}^k)_{k \in \N}$. We will recall the precise definitions of these graphs and of property $\mathcal Q$ in section \ref{subsection:Q}. Let us just say, vaguely speaking for the moment, that a Banach space $X$ has property $\mathcal Q$ if for every Lipschitz map $f$ from $([\N]^k,d_{\K}^k)$ to $X$, there exists a full subgraph $[\M]^k$ of $[\N]^k$, with $\M$ infinite subset of $\N$, on which $f$ satisfies a strong concentration phenomenon. It is then easy to see that if a Banach space $X$ has property $\mathcal Q$, then the family of graphs $([\N]^k,d_{\K}^k)_{k \in \N}$ does not equi-coarsely embed into $X$ (see the definition in section \ref{s:coarse}). One of the main results in \cite{kaltonpropertyq} is that any reflexive Banach space has property $\mathcal Q$. It then readily follows that a reflexive Banach space cannot contain a coarse copy of all separable metric spaces, or equivalently does not contain a coarse copy of the Banach space $c_0$. In fact, with a sophistication of this argument, Kalton proved an even stronger result in \cite{kaltonpropertyq}: if a separable Banach space contains a coarse copy of $c_0$, then there is an integer $k$ such that the dual of order $k$ of $X$ is non separable. In particular, a quasi-reflexive Banach space does not contain a coarse copy of $c_0$.
However, Kalton proved that the most famous example of a quasi-reflexive space, namely the James space $\J$, as well as its dual $\J^*$, fail property $\mathcal Q$.

The main purpose of this paper is to show that, although they do not obey the concentration phenomenon described by property $\mathcal Q$, neither $\J$ nor $\J^*$ equi-coarsely contains the family of graphs $([\N]^k,d_{\K}^k)_{k \in \N}$ (Corollary \ref{James}). This provides a coarse invariant, namely ``not containing equi-coarsely the Kalton graphs'', that is very close to but different from property $\mathcal Q$. This could allow to find obstructions to coarse embeddability between seemingly close Banach spaces. Our result is actually more general. We prove in Theorem \ref{general} that a quasi-reflexive Banach space $X$ such that both $X$ and $X^*$ admit an equivalent $p$-asymptotically uniformly smooth norm (see the definition in section \ref{asymptotics}), for some $p$ in $(1,\infty)$, does not equi-coarsely contain the Kalton graphs.

We conclude this note by showing that if the James tree space $\J\T$ or its predual coarsely embeds into a separable Banach space $X$, then there exists $k\in \N$ so that the dual of order $k$ of $X$ is non separable. This extends slightly Theorem 3.5 in \cite{kaltonpropertyq}.

\section{Metric notions}

\subsection{Coarse embeddings}\label{s:coarse}

Let $M$, $N$ be two metric spaces and $f \colon M \to N$ be a map. We define the compression modulus $\rho_f$ and the expansion modulus $\omega_f$ as follows. For $t\in [0,\infty)$, we set
\begin{eqnarray*}
&&\rho_f (t) = \inf \{ d_N(f(x),f(y)) \, : \, d_M(x,y) \geq t \},\\
&&\omega_f (t) = \sup \{ d_N(f(x),f(y)) \, : \, d_M(x,y) \leq t \}.
\end{eqnarray*}
We adopt the convention $\sup(\emptyset)=0$ and $\inf(\emptyset)=\infty$.
Note that for every $x,y \in M$,
$$\rho_f (d_M(x,y)) \leq d_N(f(x),f(y)) \leq \omega_f (d_M(x,y)).$$
We say that $f$ is a \emph{coarse embedding}  if $\omega_f (t) < \infty$ for every $t \in [0,+\infty)$ and $\lim_{t \to \infty} \rho_f (t) = \infty$.

Next, let $(M_i)_{i \in I}$ be a family of metric spaces. We say that the family $(M_i)_{i \in I}$ \emph{equi-coarsely embeds} into a metric space $N$ if there exist two maps $\rho, \, \omega \colon [0,+\infty) \to [0,+\infty)$ and maps $f_i \colon M_i \to N$ for $i \in I$ such that:
\begin{enumerate}[(i)]
\item $\lim_{t \to \infty} \rho(t) = \infty$,
\item $\omega(t) < \infty$ for every $t \in [0,+\infty)$,
\item $\rho(t) \leq \rho_{f_i}(t)$ and $\omega_{f_i}(t) \leq \omega(t)$ for every $i \in I$ and $t \in [0,\infty)$.
\end{enumerate}

\subsection{The Kalton interlaced graphs and property Q}
\label{subsection:Q}

For $k\in \N$  and $\M$ an infinite subset of $\N$, we put $[\M]^{\le k}=\{ S\subset \M: |S|\le k\}$,   $[\M]^{ k}=\{ S\subset \M: |S|=k\} $,  $[\M]^\omega=\{ S\subset \M: S\text{ is infinite}\} $, and $[\M]^{< \omega}=\{ S\subset \M: S\text{ is finite}\}$.
We always list the elements of some $\m$ in $[\N]^{< \omega}$ or in $[\N]^{ \omega}$ in increasing order, meaning that if we write $\m=(m_1,m_2,\ldots, m_l)$ or  $\m=(m_1,m_2,m_3, \ldots )$, we tacitly assume that $m_1<m_2<\cdots $.

For $\m=(m_1,m_2,\ldots, m_r)\in [\N]^{<\omega}$ and $\n=(n_1,n_2,\ldots, n_s)\in [\N]^{<\omega}$, we write $\m \prec \n$, if $r<s\le k$ and  $m_i=n_i$, for $i=1,2,\ldots, r$, and we write $\m\preceq \n$ if $\m \prec \n$ or $\m=\n$. Thus $\m \preceq \n$ if $\m$ is an initial segment of $\n$.

Following Kalton \cite{kaltonpropertyq}, for $\M \in [\N]^\omega$, we equip $[\M]^k$ with a graph structure by declaring $\m\neq\n\in [\M]^k$ adjacent if and only if
$$
n_1 \leq m_1 \leq n_2  \ldots \leq n_k \leq m_k\ \  \text{or}\ \  m_1 \leq n_1 \leq m_2  \ldots \leq m_k \leq n_k .
$$
For any $\m,\n \in [\M]^k$, the distance $\dk(\m,\n)$ is then defined as the shortest path distance in the graph $[\M]^k$.

\begin{remark}The distance $\dk$ is independent of the set $\M$ and therefore $[\M_1]^k$ is a metric subspace of $[\M_2]^k$ whenever $\M_1 \in [\M_2]^\omega$.
\end{remark}
This last claim is an immediate consequence of the following explicit formula for the distance.
\begin{prop}\label{p:KaltonDistanceFormula}
Let $k \in \N$ and $\M \in [\N]^\omega$.
Then $\dk(\n,\m)=d(\n,\m)$ for all $\n,\m\in [\M]^k$ where
$d(\n , \m) = \sup \{\big| |\n\cap S| - |\m\cap S| \big| \; : \; S \text{ segment of } \N \}.$
\end{prop}
\begin{proof}
It is easily seen that $d$ is a metric on $[\M]^k$.
Since $\dk$ is a graph metric on $[\M]^k$, in order to show $\dk=d$ it is enough to verify that $\dk(\n,\m)=1$ if and only if $d(\n,\m)=1$ and that $d$ is a graph metric.

For $A\subset \N$ let us denote $\indicator{A}:\N \to \set{0,1}$ the indicator function of $A$ and let us first observe the following fact.\\
\medskip
\textbf{Fact:} \emph{For every $\n,\m\in [\M]^k$,
$$d(\n,\m)=\max_{i}F(i)-\min_{i}F(i)$$
where $F(i)=F_{\n,\m}(i)=\sum_{j=1}^i\indicator{\n}(j)-\indicator{\m}(j)$ (and $F(0)=0$).}

\medskip
Indeed, we have for any segment
$S=[a,b]$ that
$$
\abs{S\cap \n}-\abs{S\cap\m}=\sum_{j\in S}\Big(\indicator{\n}(j)-\indicator{\m}(j)\Big)=F(b)-F(a-1).
$$
In particular $\max_S\big||S\cap \n|-|S\cap\m|\big|\leq \max F-\min F$.
On the other hand if $S=[a,b]$ is such that $\{F(a-1),F(b)\}=\{\max F,\min F\}$ then
$\big||S\cap \n|-|S\cap\m|\big|\geq \max F-\min F$ which finishes the proof of the fact.

It is clear that $\dk(\n,\m)=1$ if and only if $\max F-\min F=1$.
Thus it only remains to prove that $d$ is a graph metric.
Now given $\n,\m$ such that $d(\n,\m)\geq 2$ we are looking for $\lbar\in [\M]^k\setminus\set{\m,\n}$ such that $d(\m,\n)=d(\n,\lbar)+d(\lbar,\m)$.
Without loss of generality we will assume that $\max F_{\n,\m}>0$. 
Notice that the sets $\arg\max(F)$ and $\arg\min(F)$ are disjoint.
We select inductively $\{a_1<\ldots<a_p\}\subset \arg\max(F)$ and $\{b_1<\ldots<b_q\}\subset \arg\min(F)$ (with $p\geq 1$ and $q\geq 0$) with the property that

\begin{itemize}
\item $a_1=\min \arg\max(F)$,
\item For $i\ge 1$,\ $b_i=\min \left(\{n>a_i\} \cap \arg\min(F)\right)$, if this is not empty.
\item $a_{i+1}=\min\left( \{n>b_i\} \cap \arg\max(F)\right)$, if this set is not empty.
\end{itemize}
Notice that $\{a_1,\ldots,a_p\} \subset \n \setminus \m$ and $\{b_1,\ldots,b_q\} \subset \m\setminus \n$.
Notice also that either $p=q$ or $p=q+1$.
In the latter case we define $b_p:=r$ for some $r$ such that $r>a_p$ and $F(r-1)>F(r)$.
Such $r$ must exist since $F(\max\{n_k,m_k\})=0$.
Also we have $r \in \m \setminus \n$.
We will set
$$
\lbar=\n \cup \{b_1,\ldots,b_p\} \setminus \{a_1,\ldots,a_p\}.
$$
It is clear that $\lbar \in [\M]^k$.
We also have $\max F_{\lbar,\m}=\max F_{\n,\m}-1$ and $\min F_{\lbar,\m}=\min F_{\n,\m}$.
Indeed, the point $\lbar$ is constructed in such a way that when $F_{\n,\m}$ attains its maximum for the first time  (going from the left), $F_{\lbar,\m}$ is reduced by one and stays reduced by 1 until the next time the minimum of $F_{\n,\m}$ is attained (or until the point $r$) where this reduction is corrected back; and so on.
Thus $d(\lbar,\m)=d(\n,\m)-1$.
Also, since the sets $\set{a_1,\ldots,a_p}$ and $\set{b_1,\ldots,b_p}$ are interlaced we have $F_{\n,\m}-1\leq F_{\lbar,\m}\leq F_{\n,\m}$.
Therefore, since $F_{\n,\m}=F_{\n,\lbar}+F_{\lbar,\m}$, we have that $0\leq F_{\n,\lbar}\leq 1$ and so finally $d(\n,\lbar)=1$, since it is clear that $\n \neq \lbar$.
\end{proof}
Note that if $X$ is a Banach space and $f \colon ([\M]^k,d_{\K}^k) \to X$ is a map with finite expansion modulus $\omega_f$, then $\omega_f(1)$ is actually the Lipschitz constant of $f$ as $d_{\K}^k$ is a graph distance on $[\M]^k$.

In \cite{kaltonpropertyq} the property $\mathcal Q$ is defined in the setting of metric spaces. For homogeneity reasons, its definition can be simplified for Banach spaces. Let us  recall it here.

\begin{defi}
Let $X$ be a Banach space. We say that $X$ has \emph{property $\mathcal{Q}$} if there exists $C\ge 1$ such that for every $k \in \N$ and every Lipschitz map $f \colon ([\N]^k,d_{\K}^k) \to X$, there exists an infinite subset $\M$ of $\N$ such that:
$$\forall \, \overline{n},\overline{m} \in [\M]^k,\  \|f(\overline{n})-f(\overline{m})\| \leq  C\omega_f(1).$$
\end{defi}

The following proposition should be clear from the definitions. We shall however include its short proof.
\begin{prop} \label{QandGk}
Let $X$ be a Banach space. If $X$ has property $\mathcal Q$, then the family of graphs $([\N]^k,d_{\K}^k)_{k\in \N}$ does not equi-coarsely embed into $X$.
\end{prop}

\begin{proof}
Let $C\ge 1$ be given by the definition of property $\mathcal Q$. Aiming for a contradiction, assume that the family $([\N]^k,d_{\K}^k)_{k\in \N}$ equi-coarsely embeds into $X$. That is, there are maps $f_k \colon ([\N]^k,d_{\K}^k) \to X$ and two functions $\rho, \omega \colon [0,+\infty)\to[0,+\infty)$  such that $\lim_{t \to \infty} \rho (t)=\infty$ and
$$\forall k\in \N\ \ \forall t>0\ \ \rho(t) \leq \rho_{f_k}(t)\ \ \text{and}\ \ \omega_{f_k}(t) \leq \omega(t)<\infty.$$
Thus, for every $k\in \N$, there exists an infinite subset $\mathbb M_k$ of $\N$ such that $\diam(f([\M_k]^k)))\le C\omega(1)$. Since $\diam([\M_k]^k)=k$, this implies that for all $k\in \N$, $\rho(k) \le C\omega(1)$. This contradicts the fact that $\lim\limits_{t \to \infty} \rho(t)=\infty$.
\end{proof}

A concrete bi-Lipschitz copy of the metric spaces $([\N]^k,d_{\K}^k)$ in $c_0$ is given by the following proposition.

\begin{prop}\label{Kgraphsinc0} Let $(s_n)_{n=1}^\infty$ be the summing basis of $c_0$, that is\\
$s_n=\sum_{i=1}^ne_i$, where $(e_i)_{i=1}^\infty$ is the canonical basis of $c_0$.\\
For $k\in \N$, define  $f_k:([\N]^k,d_{\K}^k)\to c_0$ by $f_k(\n)=\sum_{i=1}^k s_{n_i}$. Then
$$\frac12d_{\K}^k(\n,\m)\le \|f_k(\n)-f_k(\m)\|_\infty\le d_{\K}^k(\n,\m)$$
for all $\n,\m \in [\N]^k$.
\end{prop}

\begin{proof}
Since $\dk=d$, one can show (as in the Fact in the proof of Proposition~\ref{p:KaltonDistanceFormula}) that
$\dk(\n,\m)=\max(f_k(\n)-f_k(\m))-\min(f_k(\n)-f_k(\m)).$
The result then follows easily since $\min(f_k(\n)-f_k(\m))\leq 0\leq \max(f_k(\n)-f_k(\m))$ for all $\n,\m\in [\N]^k$.
\end{proof}

\begin{remark} We already explained that $c_0$ cannot coarsely embed into any Banach space with property $\mathcal Q$ (in particular into any reflexive Banach space) and that Kalton even showed with additional arguments that if $c_0$ coarsely embeds into a separable Banach space $X$, then one of the iterated duals of $X$ has to be non separable. An inspection of his proof shows that 
the uniformly discrete metric spaces
\[
 M_k=\set{\sum_{i=1}^k s_{n_i} \times \indicator{A}: (n_1,\ldots,n_k) \in [\N]^k, A \in [\N]^\omega} \subset c_0
\]
do not equi-coarsely embed into any Banach space $X$ such that $X^{(r)}$ is separable for all $r$.
See Theorem \ref{kalton} below for more on this subject.
\end{remark}

Studying further the property $\mathcal Q$ in \cite{kaltonpropertyq}, Kalton exhibited non reflexive quasi-reflexive spaces with the property $\mathcal Q$ but showed that $\J$ and $\J^*$ fail property $\mathcal Q$. It is worth noticing that a theorem of Schoenberg  \cite{Schoenberg} implies that $\ell_1$ coarsely embeds into $\ell_2$, and therefore $\ell_1$ provides a simple example of a non-reflexive Banach space with property $\mathcal Q$.

\medskip
We conclude this section with two propositions that we state here for future reference. We start with a classical version of Ramsey's theorem.
\begin{prop}[Corollary 1.2 in \cite{Gowers}] \label{ramsey}
Let $(K,d)$ be a compact metric space, $k \in \N$ and $f \colon [\N]^k \to K$. Then for every $\ep >0$, there exists an infinite subset $\M$ of $\N$ such that $d(f(\overline{n}),f(\overline{m}))< \ep$ for every $\overline{n},\overline{m} \in [\M]^k$.
\end{prop}

For a Banach space $X$, we call \emph{tree of height $k$} in $X$ any family $(x(\n))_{\n\in[\N]^{\le k}}$, with $x(\n)\in X$. Then, if $\M \in [\N]^\omega$, $(x(\n))_{\n\in[\M]^{\le k}}$ will be called a \emph{full subtree} of $(x(\n))_{\n\in[\N]^{\le k}}$. A tree $(x^*(\n))_{\n\in[\M]^{\le k}}$ in $X^*$ is called \emph{weak$^*$-null} if for any $\n \in [\M]^{\le k-1}$, the sequence $(x^*(n_1,\ldots,n_{k-1},t))_{t>n_{k-1},t\in \M}$ is weak$^*$-null.

The next  proposition is
based on a weak$^*$-compactness argument and will be crucial for our proofs. Although the distance considered on $[\N]^k$ is different, the proof follows the same lines as Lemma 4.1 in \cite{blms}. We therefore state it now without further detail.

\begin{prop}\label{nulltree}  Let $X$ be a separable Banach space,  $k\in\N$,  and $f:([\N]^k,d_{\K}^k)\to X^*$  a Lipschitz map. Then there exist $\M \in [\N]^\omega$ and a weak$^*$-null tree $(x^*(\m))_{\m\in[\M]^{\le k}}$ in $X^*$ with $\|x^*_{\m}\| \leq \omega_f(1)$ for all $\m\in [\M]^{\leq k}\setminus\{\emptyset\}$ and so that
$$\forall \n \in [\M]^k,\ f(\n)=\sum_{i=0}^k x^*(n_1,\ldots,n_i)=\sum_{\m \preceq \n} x^*(\m).$$
\end{prop}

\section{Uniform asymptotic properties of norms and related estimates}
\label{asymptotics}

We recall the definitions that will be considered in this paper. For a Banach space $(X,\|\ \|)$ we
denote by $B_X$ the closed unit ball of $X$ and by $S_X$ its unit
sphere. The following definitions are due to V. Milman \cite{Milman} and we adopt the notation from \cite{JohnsonLindenstraussPreissSchechtman2002}. For $t\in [0,\infty)$ we define $$\overline{\rho}_X(t)=\sup_{x\in S_X}\inf_{Y}\sup_{y\in S_Y}\big(\|x+t y\|-1\big),$$
where $Y$ runs through all closed subspaces of $X$ of finite codimension. Then, the norm $\|\ \|$ is said to be {\it asymptotically uniformly smooth} (in short AUS) if
$$\lim_{t \to 0}\frac{\overline{\rho}_X(t)}{t}=0.$$
For $p\in (1,\infty)$ it is said to be {\it $p$-asymptotically uniformly smooth} (in short $p$-AUS) if there exists $c>0$ such that for all $t\in [0,\infty)$, $\overline{\rho}_X(t)\le ct^p$.

We will also need the dual modulus defined 
by
$$ \overline{\delta}_X^*(t)=\inf_{x^*\in S_{X^*}}\sup_{E}\inf_{y^*\in S_E}\big(\|x^*+ty^*\|-1\big),$$
where $E$ runs through all finite-codimensional weak$^*$-closed subspaces of~$X^*$. 
The norm of $X^*$ is said to be {\it weak$^*$ asymptotically uniformly convex} (in short AUC$^*$) if $\overline{\delta}_X^*(t)>0$ for all $t$ in $(0,\infty)$. If there exists $c>0$ and $q\in [1,\infty)$ such that for all $t\in [0,1]$ $\overline{\delta}_X^*(t)\ge ct^q$, we say that the norm of $X^*$ is $q$-AUC$^*$.
The following proposition is elementary.

\begin{prop}\label{as-sequences} Let $X$ be a Banach space. For any $t\in (0,1)$, any weakly null sequence $(x_n)_{n=1}^\infty$ in $B_{X}$ and any  $x \in S_X$ we have:
$$ \limsup_{n \rightarrow \infty} \| x+tx_n \| \leq 1 + \overline{\rho}_{X}(t).$$

For any weak$^*$-null sequence $(x^*_n)_{n=1}^\infty \subset X^*$ and for any $x^* \in X^*\setminus \set{0}$ we have
\[
 \limsup_{n\to \infty} \norm{x^*+x^*_n} \geq \norm{x^*}\left(1+\overline{\delta}_X^*\left(\frac{\limsup \norm{x_n^*}}{\norm{x^*}}\right)\right).
\]

\end{prop}

We will also need the following refinement (see Proposition 2.1 in \cite{LancienRaja2018}).

\begin{prop}\label{waus}
Let $X$ be a Banach space. 
Then the bidual norm on $X^{**}$ has the following property. 
For any $t\in (0,1)$, any weak$^*$-null sequence $(x^{**}_n)_{n=1}^\infty$ in $B_{X^{**}}$ and any  $x \in S_X$ we have:
$$ \limsup_{n \rightarrow \infty} \| x+tx^{**}_n \| \leq 1 + \overline{\rho}_{X}(t).$$
\end{prop}

\medskip Let us now recall the following classical duality result concerning these moduli (see for instance \cite{DKLR} Corollary 2.3 for a precise statement).

\begin{prop}\label{duality} Let $X$ be a Banach space. 
Then $\|\ \|_X$ is AUS if and and only if $\|\ \|_{X^*}$ is AUC$^*$.

If $p,q\in (1,\infty)$ are conjugate exponents, then $\|\ \|_X$ is $p$-AUS if and and only if $\|\ \|_{X^*}$ is $q$-AUC$^*$.
\end{prop}

We conclude this section with a list of a few classical properties of Orlicz functions and norms that are related to these moduli.
A map $\varphi:[0,\infty) \to [0,\infty)$ is called an \emph{Orlicz function} if it is continuous, non decreasing, convex and so that $\varphi(0)=0$ and $\lim_{t \to \infty}\varphi(t)=\infty$. 
The \emph{Orlicz norm} $\|\ \|_{\ell_\varphi}$, associated with  $\varphi$ is defined on $c_{00}$, the space of finitely supported sequences, as follows:
$$\forall x=(x_n)_{n=1}^\infty \in c_{00},\ \ \|x\|_{\ell_\varphi}=\inf\big\{r>0,\ \sum_{n=1}^\infty \varphi(x_n/r)\le 1\big\}.$$
The following is immediate from the definition.

\begin{lem}\label{Orlicz-lp} Let $\varphi:[0,\infty) \to [0,\infty)$ be an Orlicz function and $p\in [1,\infty)$.
\begin{enumerate}[(i)]
\item If there exists $C>0$ such that $\varphi(t)\le Ct^p$, for all $t\in [0,1]$, then there exists $A>0$ such that $\|x\|_{\ell_\varphi} \le A\|x\|_{\ell_p}$, for all $x\in c_{00}$.
\item If there exists $c>0$ such that $\varphi(t)\ge ct^p$, for all $t\in [0,1]$, then there exists $a>0$ such that $\|x\|_{\ell_\varphi} \ge a\|x\|_{\ell_p}$, for all $x\in c_{00}$.
\end{enumerate}
\end{lem}

Assume now that $\varphi:[0,\infty) \to [0,\infty)$ is an Orlicz function which is 1-Lipschitz and such that $\lim_{t\to \infty}\varphi(t)/t=1$. 
Consider for $(s,t) \in \R^2$, 
\[
N_2^\varphi(s,t)=
\begin{cases}|s|+|s|\varphi(|t|/|s|) & \text{ if }s\neq 0,\\
 \abs{t}& \text{ if }s=0.
\end{cases}
\]
Then define by induction for all $n\geq 3$:
$$\forall (s_1,\ldots,s_n)\in \R^n,\ N_n^\varphi(s_1,\ldots,s_n)=
N_2^\varphi\big(N_{n-1}^\varphi(s_1,\ldots,s_{n-1}),s_n\big).$$
The following is proved in \cite{KaltonTAMS2013} (see Lemma 4.3 and its preparation).

\begin{lemma}\label{Orlicz-Kalton}\
\begin{enumerate}[(i)]
\item For any $n \ge 2$, the function $N_n^\varphi$ is an absolute (or lattice) norm on $\R^n$, meaning that $N_n(s_1,\ldots,s_n)\le N_n(t_1,\ldots,t_n)$, whenever $|s_i|\le |t_i|$ for all $i\le n$.
\item For any $n\in \N$ and any $x\in \R^n$:
$$\frac12 \|s\|_{\ell_\varphi} \le N_n^\varphi(s) \le e\|s\|_{\ell_\varphi}.$$
\end{enumerate}
\end{lemma}

When $X$ is a Banach space, it is easy to see that $\overline{\rho}_X$ is a 1-Lipschitz Orlicz function such that $\lim_{t\to \infty}\rho(t)/t=1$.
But due to its lack of convexity, $\overline{\delta}_X^*$ is not an Orlicz function and we need to modify it. Following \cite{KaltonTAMS2013}, we define
$$\delta(t)=\int_0^t \frac{\overline{\delta}_X^*(s)}{s}\,ds.$$
It is easy to  see that $\overline{\delta}_X^*(t)/{t}$ is increasing and tends to $1$ as $t$ tends to $\infty$. Therefore, $\delta$ is an Orlicz function which is 1-Lipschitz, such that $\lim_{t\to \infty}\delta(t)/t=1$ and satisfying:
$$\forall t\in [0,\infty),\ \ \overline{\delta}_X^*(t/2) \le \delta(t) \le \overline{\delta}_X^*(t).$$
The following statement is now a direct consequence of Lemmas \ref{Orlicz-lp} and \ref{Orlicz-Kalton}.

\begin{lem}\label{Nnorm-lp} Let $X$ be a Banach space and $p\in [1,\infty)$.
\begin{enumerate}[(i)]
\item If there exists $C>0$ such that $\overline{\rho}_X(x)\le Ct^p$, for all $t\in [0,1]$, then there exists $A>0$ such that
$$\forall n\in \N\ \forall x\in \R^n,\ \ N_n^{\overline{\rho}_X}(x)\le A\|x\|_{\ell_p^n}.$$
\item If there exists $c>0$ such that $\overline{\delta}_X^*(t)\ge ct^p$, for all $t\in [0,1]$, then there exists $a>0$ such that
$$\forall n\in \N\ \forall x\in \R^n,\ \ N_n^{\delta}(x)\ge a\|x\|_{\ell_p^n}.$$
\end{enumerate}
\end{lem}

We will also use the following reformulation of Propositions~\ref{as-sequences} and~\ref{waus} in terms of the norms $N_2^\delta$ and $N_2^{\overline{\rho}_X}$.
\begin{lem}\label{l:asymptotic-Nnorm}
 Let $X$ be a Banach space. 
 \begin{enumerate}[(i)]
 \item  Let $(x_n^*) \subset X^*$ be weak$^*$-null. Then for  any $x^* \in X^*$ we have 
 \[
  \limsup_{n\to \infty} \norm{x^*+x_n^*}\geq N_2^\delta(\norm{x^*},\limsup \norm{x_n^*}).
 \]
 \item Similarly, if $(x_n^{**}) \subset X^{**}$ is weak$^*$-null and $x\in X$, then 
$$ \liminf_{n \rightarrow \infty} \| x+x^{**}_n \| \leq N_2^{\overline{\rho}_X}(\norm{x},\liminf \norm{x_n^{**}}).$$
\end{enumerate}
\end{lem}

\begin{proof}
 If $x^*=0$ there is nothing to do, so we may assume that $x^*\neq 0$.
 By application of Proposition~\ref{as-sequences} we see that 
 \[
 \begin{aligned}
  \limsup_{n\to\infty}\norm{x^*+x_n^*}&\geq \norm{x^*}\left(1+\overline{\delta}^*_X\left(\frac{\limsup\norm{x_n^*}}{\norm{x^*}}\right)\right)\\
  &\geq \norm{x^*}\left(1+\delta\left(\frac{\limsup\norm{x_n^*}}{\norm{x^*}}\right)\right)=N_2^\delta(\norm{x^*},\limsup \norm{x_n^*})
  \end{aligned}
 \]
The proof of the second claim is even simpler so we leave it to the reader.
\end{proof}

\section{The general result}

Let us first recall that a Banach space is said to be {\it quasi-reflexive} if the image of its canonical embedding into its bidual is of finite codimension in its bidual. We can now state our main result.

\begin{thm}\label{general} Let $X$ be a quasi-reflexive Banach space, let $p\in (1,\infty)$ and denote $q$ its conjugate exponent. Assume that $X$ admits an equivalent $p$-AUS norm and that $X^*$ admits an equivalent $q$-AUS norm. Then the family $([\N]^k,d_{\K}^k)_{k\in \N}$ does not equi-coarsely embed into $X^{**}$.
\end{thm}

We immediately deduce the following.

\begin{cor}\label{generalcorollary} 
Let $X$ be a quasi-reflexive Banach space, let $p\in (1,\infty)$ and denote $q$ its conjugate exponent. 
Assume that $X$ admits an equivalent $p$-AUS norm and that $X^*$ admits an equivalent $q$-AUS norm. 
Then the family $([\N]^k,d_{\K}^k)_{k\in \N}$ does not equi-coarsely embed into $X$, nor does it equi-coarsely embed into any iterated dual $X^{(r)}$ ($r\geq 0$) of $X$.
\end{cor}

\begin{proof} 
Since $X$ is quasi reflexive we infer that $X^{(r)}$ admits an equivalent $p$-AUS norm when $r$ is even and it admits an equivalent $q$-AUS norm when $r$ is odd. 
Indeed, note that when $r$ is even $X^{(r)}$ is isomorphic to $X\oplus_p F$ where $F$ is finite-dimensional (resp. $X^{(r)}\simeq X^*\oplus_q F$ when $r$ is odd).
Now it is obvious from Theorem~\ref{general} that $([\N]^k)_{k\in \N}$ do not equi-coarsely embed into $X^{(r)}$ when $r$ is even. 
When $r$ is odd, we just exchange the roles of $p$ and $q$.
\end{proof}

Before going into the detailed proof of Theorem~\ref{general} let us briefly indicate the main idea. We assume that there is an equi-coarse family of embeddings $(f_k)$ of $[\N]^k$ into $X^{**}$ with moduli $\rho$ and $\omega$.
We fix $k$ sufficiently large and observe that, up to passing to a subgraph, $f_k$ can be represented as the sum along the branches of a weak$^*$-null countably branching tree of height $k$, say $(z_{\n})_{\n\in [N]^{\leq k}}$.
Moreover the norms of the elements of this tree stabilize on each level towards values $(K_i)_{i=1}^k \subset [0,\omega(1)]$.
Applying the existence of a $q-AUS$ norm on $X^*$ one can show that $\sum_{i=1}^k K_i^p \leq c^p\omega(1)^p$ where $c$ is a constant depending only on $X$.
The benefit of this observation is twofold.
On one hand we will be able to construct two elements $\n_0,\m_0 \in [\N]^{l}$ (with $l\leq k$) such that $\sum_{i=1}^l z_{(n_1,\ldots,n_i)}-z_{(m_1,\ldots,m_i)}$ is small in norm (say less than $2c\omega(1)$) while $d_{\K}^l(\n_0,\m_0)$ is large (say $\rho(d_{\K}^l(n_0,\m_0))> 3c\omega(1)$). 
On the other hand the $p-AUS$ renormability of $X$ together with the quasi-reflexivity allows to extend these elements to elements $\n,\m \in [\N]^k$ such that $\dk(\n,\m)$ is still large and 
\[\begin{aligned}
\norm{\sum_{i=l+1}^k z_{(n_1,\ldots,n_i)}-z_{(m_1,\ldots,m_i)}} &\sim \left(\sum_{i=l+1}^k\norm{z_{(n_1,\ldots,n_i)}-z_{(m_1,\ldots,m_i)}}^p\right)^{1/p}\\ &\sim (\sum_{i=l+1}^k K_i^p)^{1/p}\leq c\omega(1)   .
\end{aligned}
\]
Eventually, summing the tree from $1$ to $k$ over the branches ending by $\n$ and $\m$ we get the desired contradiction
\[
 3c\omega(1)<\rho(\dk(\n,\m))\leq \norm{f_k(\n)-f_k(\m)}\leq 3c\omega(1).
\]

\begin{proof}[Proof of Theorem \ref{general}] Let us assume that there are two maps $\rho, \, \omega \colon [0,+\infty) \to [0,+\infty)$ and maps $f_k ([\N]^k,d_{\K}^k) \colon \to (X^{**},\|\ \|)$ for $k \in \N$ such that:
\begin{enumerate}[(i)]
\item $\lim_{t \to \infty} \rho(t) = \infty$,
\item $\omega(t) < \infty$ for every $t \in (0,+\infty)$,
\item $\rho(t) \leq \rho_{f_k}(t)$ and $\omega_{f_k}(t) \leq \omega(t)$ for every $k \in \N$ and $t \in (0,\infty)$.
\end{enumerate}
Note that all $f_k$'s are $\omega(1)$-Lipschitz for $\|\ \|$ and so $\omega(1)>0$. Since all the sets $[\N]^k$ are countable, we may and will assume that $X$ and therefore, by the quasi-reflexivity of $X$, that all its iterated duals are separable.\\
Let us fix $N\in \N$. Pick $\alpha\in \N$ such that $\alpha\ge \frac{p}{q}$ and set $k=N^{1+\alpha} \in \N$. We also fix $\eta>0$. We shall provide at the end of our proof a contradiction if $N$ is chosen large enough and $\eta$ small enough. We denote $\|\ \|$ the original norm on $X$, as well as its dual and bidual norms. Let us assume, as we may, that $\|\ \|$ is $p$-AUS on $X$. We denote its modulus of asymptotic uniform smoothness $\overline{\rho}_{\|\ \|}$ or simply $\overline{\rho}_{X}$.

\medskip
For the first step of the proof we shall exploit the existence of an equivalent $q$-AUS norm $|\ |$ on $X^*$ (we also denote $|\ |$ its dual norm on $X^{**}$). 
It is worth mentioning that if $X$ is not reflexive, $|\ |$ cannot be the dual norm of an equivalent norm on $X$ (see for instance Proposition 2.6 in \cite{CauseyLancien}). 
Assume also that there exists $b>0$ such that
\begin{equation}\label{equivalence}
\forall z \in X^{**}\ \ b\|z\|\le |z|\le \|z\|.
\end{equation}
Then we have that all $f_k$'s are also $\omega(1)$-Lipschitz for $|\ |$.\\
By Proposition \ref{duality}, we have that there exists $c>0$ such that for all $t\in [0,1]$, $\overline{\delta}_{|\ |}^*(t)\ge ct^p$. 
We denote again
$$\delta(t)=\int_0^t \frac{\overline{\delta}_{|\ |}^*(s)}{s}\, ds.$$
Recall that Lemma \ref{Nnorm-lp} ensures the existence of $a>0$ such that for all $n\in \N$, $N_n^\delta \ge 2a\|\ \|_{\ell_p^n}$.

First, using the separability of $X^*$ and Proposition \ref{nulltree}, we may assume by passing to a full subtree, that there exist a weak$^*$-null tree $(z(\m))_{\m\in[\N]^{\le k}}$ in $X^{**}$ with $|z_{\m}| \leq \omega(1)$ for all $\m\in [\N]^{\leq k}\setminus\{\emptyset\}$ and so that
$$\forall \n \in [\N]^k,\ f_k(\n)=\sum_{i=0}^k z(n_1,\ldots,n_i)=\sum_{\m \preceq \n} z(\m).$$

For $r\in \N$ we denote $E_r=\{\m=(m_1,\ldots,m_j)\in [\N]^{\le k}\setminus \{\emptyset\},\ m_j=r\}$ and $F_r=\bigcup_{u=1}^r E_u$. 
Fix a sequence $(\lambda_r)_{r=1}^\infty$ in $(0,1)$ such that $\prod_{r=1}^\infty \lambda_r > \frac12$. 
We now use Lemma~\ref{l:asymptotic-Nnorm}~(i) and the fact that $(z(\m))_{\m\in[\N]^{\le k}}$ is a weak$^*$-null tree to build inductively $n_1<\ldots<n_r$ so that for all $\n^1,\ldots,\n^L \in F_{n_{r}-1}$, for all $\eps_1,\ldots,\eps_L \in \{-1,1\}$ and all $\n \in E_{n_r}$, we have
$$\big|z(\n)+\sum_{l=1}^L\eps_l z(\n^l)\big|\ge \lambda_r N_2^{\delta}\Big(\big|\sum_{l=1}^L\eps_l z(\n^l)\big|, \big|z(\n)\big|\Big).$$
Therefore, using the fact that $N_2^\delta$ is an absolute norm and after passing to a full subtree, we may assume that for all $r_1<\cdots<r_L$ in $\N$, all $\eps_1,\ldots,\eps_L \in \{-1,1\}$ and  all $\n^1,\ldots,\n^L$ so that $\n^l \in E_{r_l}$ for $1\le l\le L$, we have
\begin{equation}\label{treelower-lp}
\big|\sum_{l=1}^L\eps_l z(\n^l)\big| \ge \frac12 N_L^\delta\big(|z(\n^1)|,\ldots,|z(\n^L)|\big)  \ge a\Big(\sum_{i=1}^L \big|z(\n^l)\big|^p\Big)^{1/p}.
\end{equation}
Assume now that $\n=(n_1,\ldots,n_k)\in \N^k$ is such that $n_1<\cdots<n_k$ are even and choose $\m=(m_1,\ldots,m_k)$ so that $n_1<m_1<\cdots<n_k<m_k$. It follows from (\ref{treelower-lp}) that
\begin{eqnarray*}
\big|f(\n)-f(\m)|=&&\Big|\sum_{i=1}^k z(n_1,\ldots,n_i)-z(m_1,\ldots,m_i)\Big| \\
&&\ge a\Big(\sum_{i=1}^k \big|z(n_1,\ldots,n_i)\big|^p+\big|z(m_1,\ldots,m_i)\big|^p\Big)^{1/p}.
\end{eqnarray*}
We now use the fact that $d_\K^k(\n,\m)=1$ and $f$ is $\omega(1)$-Lipschitz, to deduce
$$\Big(\sum_{i=1}^k \big|z(n_1,\ldots,n_i)\big|^p\Big)^{1/p}\le \frac{1}{a}\omega(1).$$
So replacing $\N$ with $2\N$ and setting $A=1/a$, we may assume that
\begin{equation}\label{endAUC}
\forall \n\in [\N]^k,\ \ \Big(\sum_{i=1}^k \big|z(n_1,\ldots,n_i)\big|^p\Big)^{1/p}\le A\omega(1).
\end{equation}

By Ramsey's theorem (Proposition \ref{ramsey}), we may also assume, after passing again to a full subtree, that for all $i\in\{1,\ldots,k\}$ there exists $K_i\in [0,\omega(1)]$ such that
$$\forall (n_1,\ldots,n_i) \in [\N]^i,\ K_i\le |z(n_1,\dots,n_i)|\le K_i+\eta.$$
The estimate \eqref{endAUC} yields
\begin{equation}\label{p-concentration}
\sum_{i=1}^k K_i^p \le A^p\omega(1)^p.
\end{equation}
Therefore, since $k=N^{1+\alpha}$, there exists $j\in \{0,N,\ldots,N(N^{\alpha}-1)\}$ such that
$$\sum_{i=j+1}^{j+N} K_i^p \le \frac{A^p\omega(1)^p}{N^\alpha}.$$
Then we deduce from H\"{o}lder's inequality that
\begin{equation}\label{Holder}
\sum_{i=j+1}^{j+N} K_i \le N^{1/q}\frac{A\omega(1)}{N^{\alpha/p}} \le A\omega(1).
\end{equation}

We now use the assumption that $X$ is quasi-reflexive, so that $X^{**}=X \oplus F$, where $F$ is of finite dimension. Thus, for each $(n_1,\dots,n_i) \in [\N]^{\le k}$, we can decompose $z(n_1,\dots,n_i)=x(n_1,\dots,n_i)+e(n_1,\dots,n_i)$, with $x(n_1,\dots,n_i)\in X$ and $e(n_1,\dots,n_i)\in F$. Then, the compactness of bounded sets in $F$ and another application of Proposition \ref{ramsey} allows us to assume, after passing to a full subtree, that
$$\forall i\in \{1,\ldots,k\}\  \forall \n,\overline{v} \in [\N]^{i},\ \ \|e(\n)-e(\overline{v})\| < \eta,$$
Which implies that for all $i\in \{1,\ldots,k\}$ and all $\n,\overline{v} \in [\N]^{i}$ we have
\begin{equation}\label{stabilize}
\big| \|z(\n)-z(\overline{v})\|-\|x(\n)-x(\overline{v})\|\big| < \eta.
\end{equation}

We are now ready for the last step of the proof, where we shall build $\m$ and $\overline{u}$ in $[\N]^k$ so that $d_\K^k(\m,\overline{u})=N$, but $|f(\m)-f(\overline{u})|$ is bounded by a constant depending only on $\omega(1)$ and on $X$. This will yield a contradiction with the fact $\lim_{N\to \infty}\rho(N)=\infty$.

First, we set $m_i=u_i=i$, for all $1\le i\le j$. 
Then, for $j+1\le i\le j+N$, we set $m_i=i$ and $u_i=i+N$. 
Finally, we shall build $m_i=u_i$ inductively, for $j+N<i\le k$. 
Note, that when this will be done, we will indeed have $d_\K^k(\m,\overline{u})=N$.

First, we obviously have
\begin{equation}\label{samedebut}
\sum_{i=1}^{j} z(m_1,\ldots,m_i)-z(u_1,\ldots,u_i)=0.
\end{equation}
The next estimate follows from (\ref{Holder}).
\begin{equation}\label{smallblock}
\Big|\sum_{i=j+1}^{j+N} z(m_1,\ldots,m_i)-z(u_1,\ldots,u_i)\Big| \le \sum_{i=j+1}^{j+N} 2(K_i+\eta) \le 3A\omega(1),
\end{equation}
if $\eta$ was initially chosen small enough.

We now select the remaining coordinates of  $\m$ and $\overline{u}$ inductively using the fact that $\|\ \|$ is $p$-AUS. 
To shorten the notation for the end of the proof, we shall now denote $x_i=x(m_1,\ldots,m_i)$, $z_i=z(m_1,\ldots,m_i)$, $x'_i=x(u_1,\ldots,u_i)$ and $z'_i=z(u_1,\ldots,u_i)$. 
First, we simply set $m_{j+N+1}=u_{j+N+1}=j+2N+1$. 
We now use the fact that the tree $(z(\m))_{\m \in [\N]^{\le k}}$ is weak$^*$-null and Lemma~\ref{l:asymptotic-Nnorm}~(ii) to find $m_{j+N+2}=u_{j+N+2}>j+2N+1$ such that
\begin{multline*}
\|x_{j+N+1}-x'_{j+N+1}+z_{j+N+2}-z'_{j+N+2}\| \\
\le N_2^{\overline{\rho}_{X}}\big(\|x_{j+N+1}-x'_{j+N+1}\|,
\|z_{j+N+2}-z'_{j+N+2}\|\big)+\eta
\end{multline*}
It follows from (\ref{stabilize}) that
\begin{multline*}
\|z_{j+N+1}-z'_{j+N+1}+z_{j+N+2}-z'_{j+N+2}\| \\
\begin{aligned}
&\le N_2^{\overline{\rho}_{X}}\big(\|z_{j+N+1}-z'_{j+N+1}\|+\eta,
\|z_{j+N+2}-z'_{j+N+2}\|\big)+2\eta\\
&\le N_2^{\overline{\rho}_{X}}\Big(\frac{2}{b}\big(K_{j+N+1}+\eta\big)+\eta,\frac{2}{b}\big(K_{j+N+2}+\eta\big)\Big)+2\eta.
\end{aligned}
\end{multline*}
Similarly, we can inductively find $m_{j+N+2}=u_{j+N+2}<\cdots<m_{k}=u_{k}$ such that,
$$
\Big\|\sum_{i=j+N+1}^k (z_i-z'_i)\Big\| \le \frac{2}{b}N_{k-j-N}^{\overline{\rho}_{X}}\big(K_{j+N+1},\ldots,K_k)+\omega(1)
$$
provided $\eta$ is chosen small enough.
Since Lemma \ref{Nnorm-lp} ensures the existence of $C>0$ such that $N_n^{\overline{\rho}_{X}}\le C\|\ \|_{\ell_p^n}$ for all $n\in \N$ the above inequality yields 
$$\Big\|\sum_{i=j+N+1}^k (z_i-z'_i)\Big\| \le \frac{2C}{b}\Big(\sum_{i=j+N+1}^k K_i^p\Big)^{1/p} +\omega(1)\le \Big(\frac{2CA}{b}+1\Big)\omega(1).$$
Finally, combining the above estimate with (\ref{samedebut}) and (\ref{smallblock}), we get that
$$\|f(\m)-f(\overline{u})\|\le \frac{3A+2CA+b}{b}\omega(1).$$
As announced at the beginning of the proof, this yields a contradiction if $N$ was initially chosen, as it was possible, so that $\rho(N)>\frac{3A+2CA+b}{b}\omega(1).$
\end{proof}

Unlike reflexivity, quasi-reflexivity itself is not enough to prevent the Kalton graphs from embedding into a Banach space. We thank P. Motakis for showing us the next example.

\begin{prop}[Motakis]\label{Pavlos} There exists a quasi-reflexive Banach space $X$ such that the family of graphs $([\N]^k,d_\K^k)_{k\in \N}$ equi-Lipschitz embeds into $X$.
\end{prop}

\begin{proof} The proof relies on the existence of a quasi-reflexive Banach space $X$ of order one which admits a spreading model, generated by a basis of $X$ that is equivalent to the summing basis $(s_n)_{n=1}^\infty$ of $c_0$. This is shown in \cite{FOSZ} (Proposition 3.2) and based on a construction given in \cite{BHO}. We refer the reader to \cite{BeauzLap} for the necessary definitions. Consequently, there exists a sequence $(x_n)_{n=1}^\infty$ in $S_X$ and constants $A,B>0$ such that for all $k\le n_1<\cdots<n_k$ and all $\eps_1,\ldots,\eps_k$ in $\{-1,0,1\}$ one has
\begin{equation}\label{spreadingsumming}
A\big\|\sum_{i=1}^k \eps_i s_i\big\|_{c_0} \le \big\|\sum_{i=1}^k \eps_i x_{n_i}\big\|_X \le B\big\|\sum_{i=1}^k \eps_i s_i\big\|_{c_0}.
\end{equation}
For $k\in \N$ and $\n=(n_1,\ldots,n_k) \in [\N]^k$ we define
$$g_k(\n)=\sum_{i=1}^k x_{2k+n_i}.$$
It follows easily from Proposition \ref{Kgraphsinc0}, the inequality (\ref{spreadingsumming})  and the fact that $(s_n)_{n=1}^\infty$ is a spreading sequence that 
$$\frac{A}{2}d_\K^k(\n,\m) \le \|g_k(\n)-g_k(\m)\|_X \le B d_\K^k(\n,\m)$$
for all $\n,\m \in [\N]^k$.

\end{proof}

\begin{remark} Let us mention that, more generally, it is proved in \cite{AMS} that for any conditional normalized spreading sequence $(e_n)_{n=1}^\infty$, there exists
a quasi-reflexive Banach space $X$ of order 1  with a normalized basis $(x_i)_{i=1}^\infty$ which generates $(e_n)_{n=1}^\infty$ as a spreading model.
\end{remark}

\section{The James sequence spaces}
\label{s:james}

Let $p\in (1,\infty)$. We now recall the definition and some basic properties of the James space $\J_p$. We refer the reader to \cite{albiackalton}(Section 3.4) and references therein for more details on the classical case $p=2$. The James space $\J_p$ is the real Banach space of all sequences $x=(x(n))_{n\in \N}$ of real numbers with finite $p$-variation and verifying $\lim_{n \to \infty} x(n) =0$. The space $\J_p$ is endowed with the following norm
$$\|x\|_{\J_p} = \sup  \Big \{  \big (\sum_{i=1}^{k-1} |x(p_{i+1})-x(p_i)|^p \big )^{1/p}     \; \colon \; 1 \leq p_1 < p_2 < \ldots < p_{k} \Big \}. $$
This is the historical example, constructed for $p=2$ by R.C. James in \cite{James1}, of a quasi-reflexive Banach space which is isomorphic to its bidual. In fact $\J_p^{**}$ can be seen as the space of all sequences $x=(x(n))_{n\in \N}$ of real numbers with finite $p$-variation, which is $\J_p \oplus \R \car$, where $\car$ denotes the constant sequence equal to $1$.\\
The standard unit vector basis $(e_n)_{n=1}^{\infty}$ ($e_n(i)=1$ if $i = n$ and $e_n(i)=0$ otherwise) is a monotone shrinking basis for $\J_p$. Hence, the sequence $(e_n^*)_{n=1}^{\infty}$ of the associated coordinate functionals is a basis of its dual $\J_p^*$. Then the weak$^*$ topology $\sigma(\J_p^*,\J_p)$ is easy to describe. A sequence $(x^*_n)_{n=1}^{\infty}$ in $\J_p^*$ converges to $0$ in the $\sigma(\J_p^*,\J_p)$ topology if and only if it is bounded and $\lim_{n \to \infty} x^*_n(i)=0$ for every $i \in \N$.\\
For $x\in \J_p$, we define $\supp x = \{i \in \N \, : \, x(i) \neq 0 \}$. For $x,y \in \J_p$, we denote: $x \prec y$ whenever $\max \supp x < \min \supp y$.\\
Similarly, an element $x^*$ of $\J_p^*$ will be written $x^*=\sum_{n=1}^\infty x^*(n)e_n^*$ and $\supp x^* = \{i \in \N \, : \, x^*(i) \neq 0 \}$ and we shall denote $x^* \prec y^*$ whenever $\max \supp x^* < \min \supp y^*$.

The detailed proof of the following proposition can be found in \cite{Netillard} (Proposition 2.3). This a consequence of the following fact: there exists $C\ge 1$ such that $\|\sum_{i=1}^n x_i\|_{\J_p}^p\le C\sum_{i=1}^n \|x_i\|_{\J_p}^p$, for all $x_1 \prec \ldots \prec x_n$ in $\J_p$.

\begin{prop}\label{Jsmooth} There exists an equivalent norm $|\ |$ on $\J_p$ such that its dual norm  $|\ |_*$ has the following property. For any $x^*,y^* \in J_p^*$ such that $x^* \prec y^*$, we have that
$$|x^*+y^*|_*^q\ge |x^*|_*^q+|y^*|_*^q.$$
In particular, $|\ |_*$ is $q$-AUC$^*$ for the weak$^*$ topology induced by $\J_p$ and therefore $|\ |$ is $p$-AUS on $\J_p$.
\end{prop}

There is also a natural weak$^*$ topology on $\J_p$. Indeed, the summing basis $(s_n)_{n=1}^{\infty}$ ($s_n(i)=1$ if $i \leq n$ and $s_n(i)=0$ otherwise) is a monotone and boundedly complete basis for $\J_p$. Thus, $\J_p$ is naturally isometric to a dual Banach space: $\J_p = X^*$ with $X$ being the closed linear span of the biorthogonal functionals $(e_n^* - e_{n+1}^*)_{n=1}^{\infty}$ in $\J_p^*$ associated with $(s_n)_{n=1}^{\infty}$. Note that $X=\{x^*\in \J_p^*,\ \sum_{n=1}^\infty x^*(n)=0\}$. Thus, a sequence $(x_n)_{n=1}^{\infty}$ in $\J_p$ converges to $0$ in the $\sigma(\J_p,X)$ topology if and only if it is bounded and $\lim_{n \to \infty} \big(x_n(i) - x_n(j)\big) = 0$ for every $i \neq j \in \N$. The next proposition is easy (see Proposition 2.3 in \cite{LancienPMB} for the case $p=2$).

\begin{prop}\label{Jconvex} The usual norm on $\J_p$ is p-AUC$^*$ for the weak$^*$ topology induced by $X$. In other words, the restriction to $X$ of the usual norm on $\J_p^*$ is $q$-AUS.
\end{prop}

Then, since $X$ is one codimensional in $\J_p^*$, we have that $\J_p^*$ is isomorphic to $X\oplus \R$ and therefore also admits an equivalent $q$-AUS norm.

\medskip
The above remarks combined with Corollary \ref{generalcorollary} immediately yield the following.

\begin{cor}\label{James} Let $p\in (1,\infty)$. Then, the family $([\N]^k,d_{\K}^k)_{k\in \N}$ does not equi-coarsely embed into $\J_p$, nor does it equi-coarsely embed into $\J_p^*$.
\end{cor}

\section{A Remark on the James tree space}\label{s:coarsejamestree}

Let us recall the construction of the James tree space $\J\T$. We denote $T=2^{<\omega}$ the tree of all finite sequences with coefficients in $\{0,1\}$ equipped with its natural order: for $s,t \in T$, we say that $s \le t$ if the sequence $t$ extends $s$. The set of all infinite sequences with coefficients in $\{0,1\}$ will be denoted $2^\omega$. For $s\in T$, the length of $s$ is denoted $|s|$. We call $\emph{segment}$ of $T$ any set of the form $\{s\in T,\ t\le s \le t'\}$ with $t\le t'$ in $T$. For a map $x:T\to \R$, we define
$$\|x\|_{\J\T}=\sup\Big\{\Big(\sum_{i=1}^n\big(\sum_{s\in S_i} x(s)\big)^2\Big)^{1/2} \Big\},$$
where the supremum is taken over all pairwise disjoint segments $S_1,\ldots,S_n$ of $T$. Then the \emph{James tree} space is the space $\J\T=\{x:T \to \R,\ \|x\|_{\J\T}<\infty\}$ equipped with the norm $\|\ \|_{\J\T}$. For $s\in T$, we denote $e_s:T\to \R$ defined by $e_s(t)=\delta_{s,t},\ t\in T$. If $\psi:\N \to T$ is a bijection such that $|\psi(n)|\le |\psi(m)|$ whenever $n\le m$, then $(e_{\psi(n)})_{n=1}^\infty$ is a normalized, monotone and boundedly complete basis of $\J\T$. For $s\in T$, the coordinate functional $e^*_s$ is defined by $e^*_s(x)=x(s)$, $x\in \J\T$. Then the closed linear span of $\{e^*_s,\ s\in T\}$ in $\J\T^*$ is denoted $\B$ and $\B^*$ is isometric to $\J\T$.
The space $\J\T$ was built by R.C. James in \cite{James2} to serve as the first example of a separable Banach space with non separable dual, which does not contain an isomorphic copy of $\ell_1$.

In \cite{kaltonpropertyq} it is shown that if a Banach space $X$ coarsely contains $c_0$ then there exists $k\in \N$ such that $X^{(k)}$, the dual of order $k$ of $X$, is non separable. A close look at the proof of Theorem 3.5 in \cite{kaltonpropertyq} allows to state the following.

\begin{thm}[Kalton]\label{kalton} Let $X$ and $Y$ be two Banach spaces such that $X$ coarsely embeds into $Y$. Assume moreover that there exist an uncountable set $I$ and for every $i\in I$ and $k\in \N$, a 1-Lipschitz map $f_i^k:([\N]^k,d_{\K}^k))\to X$ such that
$$\lim_{k\to \infty}\ \inf_{i\neq j \in I}\ \inf_{\M\in [\N]^\omega}\ \sup_{\n\in [\M]^k} \|f_i^k(\n)-f_j^k(\n)\|=\infty.$$
Then there exists $r\in \N$ such that $Y^{(r)}$ is not separable.
\end{thm}

As an application, we can show the following.

\begin{thm}\label{kalton+} Let $Y$ be a Banach space such that $\B$ or $\J\T$ coarsely embeds into $Y$. Then there exists $r\in \N$ such that $Y^{(r)}$ is not separable.
\end{thm}

\begin{proof} For $\sigma \in 2^\omega$, we denote $\sigma_{|n}=(\sigma_1,\ldots,\sigma_n)$.
Then, for $k\in \N$, we define $f_\sigma^k: [\N]^k \to \B$ as follows. For $\n=(n_1,\ldots,n_k)\in [\N]^k$ let
$$f_\sigma^k(\n)=\frac{1}{\sqrt{k}}\sum_{i=1}^k \sum_{s \le \sigma_{|n_i}} e^*_s.$$
Assume for instance that $n_1\le m_1 \le \cdots n_k\le m_k$. Then we can write
$$f_\sigma^k(\m)-f_\sigma^k(\n)= \frac{1}{\sqrt{k}} \sum_{i=1}^k \sum_{s \in S_i}e^*_s,$$
where $S_1,\ldots,S_k$ are pairwise disjoint segments in $T$. 
Note that for any segment $S_i$ the sum $\sum_{s \in S_i}e^*_s$ belongs to the unit ball of $\J\T^*$. It then follows from the Cauchy-Schwarz inequality that $f_\sigma^k$ is 1-Lipschitz on $([\N]^k,d_{\K}^k)$. 
Assume now that $\sigma\neq \tau \in 2^\omega$. Pick $r\in \N$ such that $\sigma_r\neq \tau_r$. Then for any $\M\in [\N]^\omega$ and any $\n=(n_1,\dots,n_k)\in [\M]^k$ with $n_1\ge r$, we have
$$\|f_\sigma^k(\n)-f_\tau^k(\n)\|_{\B}\ge \Big|\langle f_\sigma^k(\n)-f_\tau^k(\n), e_{\sigma_{|n_1}} \rangle \Big| \ge \sqrt{k}.$$
By Theorem \ref{kalton} and the uncountability of $2^\omega$, this finishes our proof for $\B$.

For $\sigma \in 2^\omega$ and $k\in \N$ define now $g_\sigma^k: [\N]^k \to \J\T$ by
$$\forall \n=(n_1,\ldots,n_k)\in [\N]^k,\ g_\sigma^k(\n)=\frac{1}{\sqrt{2k}} \sum_{i=1}^k e_{\sigma_{|n_i}}.$$
It is easily checked that $g_\sigma^k$ is 1-Lipschitz on $([\N]^k,d_{\K}^k)$.
Assume now that $\sigma\neq \tau \in 2^\omega$. Pick $r\in \N$ such that $\sigma_r\neq \tau_r$. Then for any $\M\in [\N]^\omega$ and any $\n=(n_1,\dots,n_k)\in [\M]^k$ with $n_1\ge r$, denote $S=\{s\in T, \sigma_{|n_1} \le s \le \sigma_{|n_k}\}$. The set $S$ is a segment in $T$ and $x^*=\sum_{s\in S}e^*_s$ is in the unit ball of $\J\T^*$. Therefore
$$\|g_\sigma^k(\n)-g_\tau^k(\n)\|_{\J\T}\ge \langle g_\sigma^k(\n)-g_\tau^k(\n), x^* \rangle \ge \frac{\sqrt{k}}{\sqrt{2}}.$$
This concludes our proof for $\J\T$.
\end{proof}

\noindent{\bf Acknowledgements.} We are grateful to P. Motakis for allowing us to include Proposition \ref{Pavlos} and for very useful discussions. We also wish to mention that the formula describing the interlaced distance given in Proposition \ref{p:KaltonDistanceFormula} has been independently obtained by 
F. Baudier, P. Motakis and Th. Schlumprecht.

\end{document}